\documentclass[leqno,11pt]{article}
\usepackage{amssymb,amsmath,mathabx}

\usepackage{pgf,ifpdf,graphics,xcolor}

\usepackage{amsfonts,amssymb,euscript,epsfig,color}
\usepackage[all]{xy}

\newtheorem{theo}{Theorem}[section]
\newtheorem{defi}[theo]{Definition}

\newtheorem{lem}[theo]{Lemma}
\newtheorem{prop}[theo]{Proposition}
\newtheorem{rem}[theo]{Remark}

\newenvironment{proof}{{\bf Proof.}}

\newcommand{\ggot}{\ensuremath{\mathfrak{g}}}
\newcommand{\hgot}{\ensuremath{\mathfrak{h}}}
\newcommand{\kgot}{\ensuremath{\mathfrak{k}}}

\newcommand{\tgot}{\ensuremath{\mathfrak{t}}}

\newcommand{\Acal}{\ensuremath{\mathcal{A}}}
\newcommand{\Bcal}{\ensuremath{\mathcal{B}}}

\newcommand{\Dcal}{\ensuremath{\mathcal{D}}}

\newcommand{\Lcal}{\ensuremath{\mathcal{L}}}

\newcommand{\Ocal}{\ensuremath{\mathcal{O}}}
\newcommand{\Pcal}{\ensuremath{\mathcal{P}}}
\newcommand{\Qcal}{\ensuremath{\mathcal{Q}}}
\newcommand{\Scal}{\ensuremath{\mathcal{S}}}
\newcommand{\Ucal}{\ensuremath{\mathcal{U}}}

\newcommand{\Xcal}{\ensuremath{\mathcal{X}}}

\newcommand{\Pbb}{\ensuremath{\mathbb{P}}}
\newcommand{\Z}{\ensuremath{\mathbb{Z}}}
\newcommand{\C}{\ensuremath{\mathbb{C}}}

\newcommand{\R}{\ensuremath{\mathbb{R}}}

\newcommand{\mm}{\ensuremath{\hbox{\rm m}}}

\newcommand{\croc}{\ensuremath{\hookrightarrow}}

\newcommand{\qfor}{\ensuremath{\mathcal{Q}^{-\infty}}}

\newcommand{\QS}{\ensuremath{\mathrm{Q}^{\mathrm{spin}}}}

\newcommand{\spinc}{\ensuremath{{\rm spin}^{c}}}

\def \T {{\rm T}}

\def \what {\widehat}

\def \clif {{\bf c}}

\begin{document}

\title{Formal Geometric Quantization III \\ Functoriality in the $\spinc$ setting}

\author{Paul-Emile PARADAN
\footnote{Institut Montpelli\'erain Alexander Grothendieck,
Universit\'e de Montpellier, CNRS \texttt{paul-emile.paradan@umontpellier.fr}}}


\maketitle

{\small
\tableofcontents}

\section{Introduction}

Let $(M,\Scal)$ be a $K$-manifold of even dimension, oriented, and equipped with a $K$-equivariant $\spinc$ bundle $\Scal$. 
The orientation induces a decomposition $\Scal=\Scal^+\oplus \Scal^-$, and the corresponding $\spinc$ Dirac operator
is a first order elliptic operator $\Dcal_\Scal: \Gamma(M,\Scal^+)\to \Gamma(M,\Scal^-)$ \cite{Atiyah-Singer-2,B-G-V,Duistermaat96}. 

When $M$ is compact, an important invariant is the equivariant index $\Qcal_K(M,\Scal)\in R(K)$ of the operator $\Dcal_\Scal$, that can be understood as the quantization of the data $(M,\Scal,K)$ \cite{Atiyah-Hirzebruch70,Hattori78}.

The determinant line bundle of the $\spinc$-bundle $\Scal$ is the line bundle $\det(\Scal)\to M$ defined by the relation
$$
\det(\Scal):=\hom_{\mathrm{Cl}}(\overline{\mathcal{S}},\mathcal{S})
$$
where $\overline{\mathcal{S}}$ is the Clifford module with opposite complex structure (see \cite{pep-vergne-acta}).

The choice of an invariant Hermitian connection $\nabla$ on $\det(\Scal)$ determines an equivariant map $\Phi_{\Scal}: M\to \kgot^{*}$ and a 2-form
$\Omega_\Scal$ on $M$ by means of the  Kostant relations
\begin{equation}\label{eq:kostant-L}
    \Lcal(X)-\nabla_{X_M}=2i \langle\Phi_\Scal,X\rangle\quad \mathrm{and} \quad \nabla^2= -2i \Omega_\Scal
\end{equation}
for every $X\in\kgot$. Here $\Lcal(X)$ denotes the infinitesimal action on the sections of $\det(\Scal)$.
We say that $\Phi_\Scal$ is the {\em moment map} for $\Scal$ (it depends however of the choice of a connection).

Assume now that $M$ is non-compact but that the moment map $\Phi_\Scal$ is {\em proper}. In this case the formal geometric 
quantization  of $(M,\Scal,K)$ is well-defined:
$$
\qfor_K(M,\Scal)\in \hat{R}(K)
$$
as an index localized on the zeros of the Kirwan vector field \cite{pep-ENS,pep-pacific,Ma-Zhang14,Hochs-Song-duke}. We will explain the construction in section \ref{sec:QR-non-compact}.

Consider now a closed connected subgroup $H\subset K$. Let ${\rm p}:\kgot^*\to \hgot^*$ be the canonical projection. The map 
${\rm p}\circ \Phi_\Scal$ corresponds to the moment map for $\Scal$ relative to the $H$-action.

The main result of our note is the following

\begin{theo}\label{theo-principal-introduction}
Suppose that ${\rm p}\circ \Phi_\Scal$ is proper. Then the following holds:
\begin{itemize}
\item The $K$-module $\qfor_K(M,\Scal)$ is $H$-admissible.
\item We have $\qfor_K(M,\Scal)\vert_H=\qfor_H(M,\Scal)$.
\end{itemize}
\end{theo}

We obtained a similar result in the symplectic setting in \cite{pep-IF}. Here our method uses the compactifications of reductive groups \`a la 
de Concini-Procesi and the multiplicative property of the functor $\qfor_K$ that has been proved recently by 
Hochs and  Song \cite{Hochs-Song-duke}.

\begin{center}
\bf Notations
\end{center}

Throughout the paper :
\begin{itemize}
\item $K$ denotes a compact connected Lie group with Lie algebra $\kgot$.
\item We denote by $R(K)$ the representation ring of $K$ : an element $E\in R(K)$ can be represented
as a finite sum $E=\sum_{\mu\in\what{K}}\mm_\mu \pi_\mu$, with $\mm_\mu\in\Z$. The multiplicity
of the trivial representation is denoted $[E]^K$.

\item We denote by $\hat R(K)$ the space of $\Z$-valued functions on $\widehat{K}$. An element $E\in
    \hat R(K)$ can be represented
as an infinite sum $E=\sum_{\mu\in\what{K}}\mm(\mu) \pi_\mu$, with $\mm(\mu)\in\Z$. 

\item An element $\xi\in \kgot^*$  is called regular if the stabilizer subgroup $K_\xi:=\{k\in K, k\cdot\xi=\xi\}$ is a maximal torus of $K$.

\item When $K$ acts on a manifold $M$, we denote $X_M(m):=\frac{d}{dt}\vert_{t=0} e^{-tX}\cdot m$
the vector field generated by $-X\in \kgot$. Sometimes we will also use the notation $X_M(m)=-X\cdot m$. 
The set of zeroes of the vector field $X_M$ is denoted $M^X$.

\end{itemize}

\section{The $[Q,R]=0$ theorem in the $\spinc$ setting}\label{sec:QR-compact}

In this section we suppose that $M$ is compact and we recall the results of \cite{pep-vergne-acta} concerning the multiplicities of $\Qcal_K(M,\Scal)\in R(K)$.

For any $\xi\in\kgot^*$, we denote $\kgot_\xi$ the Lie algebra of the stabilizer subgroup $K_\xi$. A coadjoint 
orbit $\Pcal=K\eta$ is of type $(\kgot_\xi)$ if the conjugacy classes $(\kgot_\eta)$ and $(\kgot_\xi)$ are equal.

\begin{defi}
Let $(\kgot_M)$ be the generic infinitesimal stabilizer for the $K$-action on $M$. We says that the $K$-action on $M$ is {\em nice} if there exists $\xi\in\kgot^*$ such that 
\begin{equation}\label{eq:nice}
([\kgot_M,\kgot_M])=([\kgot_\xi,\kgot_\xi]).
\end{equation} 
\end{defi}

The first result of \cite{pep-vergne-acta} is the following

\begin{theo}
If the $K$-action on $M$ is not nice, then $\Qcal_K(M,\Scal)=0$ for any $\spinc$ bundle $\Scal$.
\end{theo}

We suppose now that the $K$-action on $M$ is nice. The conjugacy class $(\kgot_\xi)$ satisfying (\ref{eq:nice}) is unique and is denoted $(\hgot_M)$ (see Lemma 3 in \cite{pep-vergne:magic}).

\begin{defi}A coadjoint orbit  $\Pcal\subset \kgot^*$ is admissible if $\Pcal$ carries a $\spinc$-bundle $\Scal_\Pcal$ such that
the corresponding moment map  is the inclusion $\Pcal\croc \kgot^*$. We denote simply by
$\QS_K(\Pcal)$ the element $\Qcal_K(\Pcal,\Scal_\Pcal)\in R(K)$.
\end{defi}

We can check easily \cite{pep-vergne:magic} that  $\QS_K(\Pcal)$ is either $0$ or an irreducible representation
of $K$, and that the map
$$
\Ocal\mapsto \pi_\Ocal^K:=\QS_K(\Ocal)
$$
defines a bijection between the regular admissible orbits and the dual $\widehat{K}$.
When $\Ocal$ is a regular admissible orbit, an admissible coadjoint orbit  $\Pcal$
is called an ancestor of $\Ocal$ (or a $K$-ancestor of $\pi^K_\Ocal$)
 if $\QS_K(\Pcal)=\pi^K_\Ocal$.

\medskip

Denote by $\Acal((\hgot_M))$ the set of admissible orbits  of type $(\hgot_M)$. The following important fact is proved in section 5 
of \cite{pep-vergne-acta}.

\begin{prop}
Let $\Pcal\in \Acal((\hgot_M))$. 
\begin{itemize}
\item If $\Pcal$ belongs to the set of regular values of $\Phi_\Scal$, the reduced space 
$$
M_\Pcal=\Phi^{-1}_\Scal(\Pcal)/K
$$
is an oriented orbifold equipped with a $\spinc$ bundle. The index of the corresponding Dirac operator on the orbifold $M_\Pcal$ is denoted 
$\QS(M_\Pcal)\in \Z$ \cite{Kawasaki81}.

\item In general, the $\spinc$ index $\QS(M_\Pcal)\in \Z$ associated to the (possibly singular) reduced space $M_\Pcal$ is defined 
by a deformation procedure.
\end{itemize}

\end{prop}

The $[Q,R]=0$ Theorem in the $\spinc$ setting takes the following form.

\begin{theo}[\cite{pep-vergne-acta}]\label{theo:QR-compact} 

Let $\Ocal$ be a regular admissible orbit.

The multiplicity of the representation $\pi^K_\Ocal$ in $\Qcal_K(M,\Scal)$ is equal to
$$
\sum_{\Pcal}  \QS(M_{\Pcal})
$$
where the sum runs over the ancestors of $\Ocal$ of type $(\hgot_M)$. In other words
$$
\Qcal_K(M,\Scal)=\sum_{\Pcal\in \Acal((\hgot_M))}\QS(M_{\Pcal})\,\QS_K(\Pcal).
$$

\end{theo}

\section{Formal geometric quantization in the $\spinc$ setting}\label{sec:QR-non-compact}

In this section the manifold $M$ is not necessarily compact, but the moment map $\Phi_\Scal$ is supposed to be proper.

We  choose  an invariant scalar product in $\kgot^*$ that provides an identification
$\kgot\simeq\kgot^*$.

\begin{defi}\label{defi:kir}
$\bullet$ The {\em Kirwan vector field} associated to $\Phi_\Scal$ is defined by
\begin{equation}\label{eq-kappa}
    \kappa_{\Scal}(m)= -\Phi_\Scal(m)\cdot m, \quad m\in M.
\end{equation}

$\bullet$ We denote by $Z_\Scal$ the set of zeroes of $\kappa_{\Scal}$. Thus $Z_\Scal$ is a $K$-invariant closed subset of $M$.
\end{defi}

The set $Z_\Scal$, which is not necessarily smooth, admits an easy description. Choose a Weyl chamber $\mathfrak{t}^*_+\subset \mathfrak{t}^*$ in the dual 
of the Lie algebra of a maximal torus $T$ of $K$. We see that 
\begin{equation}\label{eq=Z-Phi-beta}
Z_\Scal=\coprod_{\beta\in\Bcal_\Scal} Z_\beta
\end{equation}
where $Z_\beta$ corresponds to the compact set $K(M^\beta\cap\Phi_\Scal^{-1}(\beta))$, and $\Bcal_\Scal=\Phi_\Scal(Z_\Scal)\cap \tgot^*_+$.  
The properness of $\Phi_\Scal$ insures that for any compact subset $C\subset \tgot^*$ the intersection $\Bcal_\Scal\cap C$ is finite.

The principal symbol of the Dirac operator $D_\Scal$ is the bundle map
$\sigma(M,\Scal)\in \Gamma(\T^* M, \hom(\Scal^+,\Scal^-))$ defined by the Clifford action
$$\sigma(M,\Scal)(m,\nu)=\clif_{m}(\tilde{\nu}): \Scal\vert_m^+\to \Scal\vert_m^-.$$
where $\nu\in \T^* M\simeq \tilde{\nu}\in \T M$ is an identification associated to an invariant Riemannian metric on $M$.
\begin{defi}\label{def:pushed-sigma}
The symbol  $\sigma(M,\Scal,\Phi_\Scal)$ shifted by the vector field $\kappa_{\Scal}$ is the
symbol on $M$ defined by
$$
\sigma(M,\Scal,\Phi_\Scal)(m,\nu)=\sigma(M,\Scal)(m,\tilde{\nu}-\kappa_\Scal(m))
$$
for any $(m,\nu)\in\T^* M$.
\end{defi}

For any $K$-invariant open subset $\Ucal\subset M$ such that $\Ucal\cap Z_\Scal$ is compact in $M$, we see that the restriction
$\sigma(M,\Scal,\Phi_\Scal)\vert_\Ucal$ is a transversally elliptic symbol on $\Ucal$, and so its equivariant index is a well defined element in
$\hat{R}(K)$ (see \cite{Atiyah74,pep-vergne:witten}).

Thus we can define the following localized equivariant indices.

\begin{defi}\label{def:indice-localise}
\begin{itemize}
\item A closed invariant subset $Z\subset Z_\Scal$ is called a {\em component} of $Z_\Scal$ if it is a union of connected components of $Z_\Scal$.

\item  If $Z$ is a {\em compact component} of $Z_\Scal$, we denote by
$$
\Qcal_K(M,\Scal,Z)\ \in\ \hat{R}(K)
$$
the equivariant index of $\sigma(M,\Scal,\Phi)\vert_\Ucal$ where $\Ucal$ is an invariant neighbourhood of $Z$
so that $\Ucal\cap Z_\Scal=Z$.
\end{itemize}
\end{defi}

By definition, $Z=\emptyset$ is a component of $Z_\Scal$ and $\Qcal_K(M,\Scal,\emptyset)=0$. For any $\beta\in\Bcal_\Scal$, $Z_\beta$ is a 
compact component of  $Z_\Scal$.

When the manifold $M$ is compact, the set $\Bcal_\Scal$ is finite and we have the decomposition
$$
\Qcal_K(M,\Scal)=\sum_{\beta\in\Bcal_\Scal}\Qcal_K(M,\Scal,Z_\beta).
$$
See \cite{pep-RR,pep-vergne:witten}.

When the manifold $M$ is not compact, but the moment map $\Phi_\Scal$ is proper, we can defined
$$
\qfor_K(M,\Scal):=\sum_{\beta\in\Bcal_\Scal}\Qcal_K(M,\Scal,Z_\beta)
$$
The sum of the right hand side is not necessarily finite but it converges in $\hat{R}(K)$ (see \cite{pep-pacific,Ma-Zhang14,Hochs-Song-duke}). 
We call $\qfor_K(M,\Scal)\in \hat{R}(K)$ the formal geometric quantization of the data $(M,\Scal,\Phi_\Scal, K)$.

Hochs and Song prove the following important property concerning the functoriality of $\qfor_K$ relatively to the product of manifolds.

\medskip

\begin{theo}[\cite{Hochs-Song-duke}]\label{theo:hochs-song}
Let $(M,\Scal)$ be a $\spinc$ $K$-manifold with a proper moment map $\Phi_\Scal$. 
Let $(P,\Scal_P)$ be a compact $\spinc$ $K$-manifold (even dimensional and oriented). Then the $\spinc$ manifold
$(M\times P,\Scal\boxtimes \Scal_P)$ admits a proper moment map and we have the following equality in $\hat{R}(K)$:
$$
\qfor_K(M\times P,\Scal\boxtimes \Scal_P) = \qfor_K(M,\Scal)\otimes \Qcal_K(P,\Scal_P).
$$
\end{theo}

\medskip

With Theorem \ref{theo:hochs-song} in hand we can compute the multiplicities of $\qfor_K(M,\Scal)$ like in the compact setting by using the shifting trick. 

Let $\Ocal$ be an admissible regular orbit of $K$. We denote by $[\pi^K_\Ocal:\qfor_K(M,\Scal)]$ the multiplicity of $\pi^K_\Ocal$ in 
$\qfor_K(M,\Scal)\in \hat{R}(K)$. Let $\Ocal^*$ be the admissible orbit $-\Ocal$: we have $\Qcal_K(\Ocal^*,\Scal_{\Ocal^*})=(\pi^K_\Ocal)^*$.
Thanks to Theorem \ref{theo:hochs-song} we get 
\begin{eqnarray*}
[\pi^K_\Ocal:\qfor_K(M,\Scal)]&=&\left[\qfor_K(M,\Scal)\otimes \Qcal_K(\Ocal^*,\Scal_{\Ocal^*})\right]^K\\
&=& \left[\qfor_K(M\times \Ocal^*,\Scal\otimes\Scal_{\Ocal^*})\right]^K.
\end{eqnarray*}

We consider the product $M\times \Ocal^*$ equipped with the  $\spinc$-bundle $\Scal\otimes\Scal_{\Ocal^*}$.
 The corresponding  moment map is $\Phi_{\Scal\otimes\Scal_{\Ocal^*}}(m,\xi)=\Phi_\Scal(m)+\xi$.
We use the simplified notation $\Phi_\Ocal$ for $\Phi_{\Scal\otimes\Scal_{\Ocal^*}}$, $\kappa_\Ocal$  for the
 corresponding Kirwan vector field on $M\times \Ocal^*$, and $Z_\Ocal:=\{\kappa_\Ocal=0\}$. 
 
 In \cite{pep-vergne-acta}, we introduced a locally constant function $d_\Ocal: Z_\Ocal\to \R$, and we denote 
 $Z_\Ocal^{=0}=\{d_\Ocal=0\}$. Using the localization\footnote{In \cite{pep-vergne-acta} we work in the compact setting, but 
 exactly the same proof works in the non compact setting,  as noticed by Hochs and Song \cite{Hochs-Song-duke}.} done in \cite{pep-vergne-acta}[section 4.5], we get that 
\begin{equation}\label{eq:Z-Ocal-0}
[\pi^K_\Ocal:\qfor_K(M,\Scal)]=
 \left[\qfor_K(M\times \Ocal^*,\Scal\otimes\Scal_{\Ocal^*},Z_\Ocal^{=0})\right]^K
\end{equation}
 
 Finally we obtain the same result like in the compact setting:
 \begin{itemize}
\item If the $K$ action on $M$ is not nice, $Z_\Ocal^{=0}=\emptyset$ and then the multiplicity 
$[\pi^K_\Ocal:\qfor_K(M,\Scal)]^K$ vanishes for any 
 regular admissible orbit $\Ocal$.
 \item If the $K$ action on $M$ is nice, we have 
$[\pi^K_\Ocal:\qfor_K(M,\Scal)]^K=\sum_{\Pcal}  \QS(M_{\Pcal})$ where the sum runs over the ancestors of $\Ocal$ of type $(\hgot_M)$.
 \end{itemize}
 
 In other words, 
 \begin{itemize}
\item if the $K$ action on $M$ is not nice, then $\qfor_K(M,\Scal)=0$, 

\item if the $K$ action on $M$ is nice, we have 
$$
\qfor_K(M,\Scal)=\sum_{\Pcal\in \Acal((\hgot_M))}\QS(M_{\Pcal})\,\QS_K(\Pcal).
$$
 \end{itemize}

 \begin{rem}\label{rem:abelian-stabilizer}
 We will use a particular case of identity (\ref{eq:Z-Ocal-0}) when the generic infinitesimal stabilizer of the $K$-action on $M$ is {\em abelian}, i.e. 
 $([\kgot_M,\kgot_M])=0$. In this case $Z_\Ocal^{=0}=\{\Phi_\Ocal=0\}$ and then
\begin{eqnarray*}
[\pi^K_\Ocal:\qfor_K(M,\Scal)]^K&=&
 \left[\qfor_K(M\times \Ocal^*,\Scal\otimes\Scal_{\Ocal^*},\{\Phi_\Ocal=0\})\right]^K\\
 &=& \QS(M_{\Ocal}).
\end{eqnarray*}
 \end{rem}

\section{Functoriality relatively to a subgroup}

We come back to the setting of $K$-manifold $M$, even dimensional and oriented, equipped with an equivariant $\spinc$ bundle $\Scal$. We suppose that for some choice of connection on $\det(\Scal)$ the corresponding moment map $\Phi_\Scal$ is proper. As explained earlier, 
the formal geometric quantization  of $(M,\Scal,K)$ is well-defined :
$\qfor_K(M,\Scal)\in \hat{R}(K)$.

Consider now a closed connected subgroup $H\subset K$. Let ${\rm p}:\kgot^*\to \hgot^*$ be the canonical projection. The map 
${\rm p}\circ \Phi_\Scal$ correspond to the moment map for the $\spinc$ bundle $\Scal$ relative to the $H$-action.

This section is dedicated to the proof of our main result.

\begin{theo}\label{theo-principal}
Suppose that ${\rm p}\circ \Phi_\Scal$ is proper. Then the $K$-module $\qfor_K(M,\Scal)$ is $H$-admissible, and
\begin{equation}\label{eq:main}
\qfor_K(M,\Scal)\vert_H=\qfor_H(M,\Scal).
\end{equation}

\end{theo}

We start with the following

\begin{lem}\label{lem:simple}
$\bullet$ $\qfor_K(M,\Scal)$ is $H$-admissible when ${\rm p}\circ \Phi_\Scal$ is proper.

$\bullet$ It is sufficient to prove (\ref{eq:main}) for manifolds with abelian generic infinitesimal stabilizers.
\end{lem}

\begin{proof} We have $\qfor_K(M,\Scal)=\sum_\Pcal \QS(M_\Pcal) \QS_K(\Pcal)$ where the sum runs over the admissible orbits of type 
$(\hgot_M)$. 

Thanks to the $[Q,R]=0$ Theorem we know that $\QS_K(\Pcal)\vert_H= \QS_H(\Pcal)= \sum_{\Pcal'}\QS(\Pcal_{\Pcal'}) \QS_H(\Pcal')$, where 
$\Pcal_{\Pcal'}= \Pcal\cap {\rm p}^{-1}(\Pcal')/H$ is the reduction of the $K$-coadjoint orbit $\Pcal$ relatively to $H$-coadjoint orbit $\Pcal'$. 

Hence $\qfor_K(M,\Scal)$ is $H$-admissible if for any $H$-coadjoint orbit $\Pcal'$, the sum 
$$
\sum_\Pcal \QS(M_\Pcal)\QS(\Pcal_{\Pcal'}) 
$$
admits only a finite number of non-zero terms. We see that $\QS(\Pcal_{\Pcal'})\neq 0$ only if $\Pcal'\subset{\rm p}(\Pcal)$ and 
$\QS(M_\Pcal)\neq 0$ only if $\Phi_\Scal^{-1}(\Pcal)\neq\emptyset$. Finally $\QS(M_\Pcal)\QS(\Pcal_{\Pcal'})\neq 0$ only if 
$$
\Pcal\in K\Phi_\Scal\left(({\rm p}\circ \Phi_\Scal)^{-1}(\Pcal')\right).
$$
Since ${\rm p}\circ \Phi_\Scal$ is proper, we have only a finite number of $K$-admissible orbits contained in the compact set 
$K\Phi_\Scal\left(({\rm p}\circ \Phi_\Scal)^{-1}(\Pcal')\right)$. The first point is proved.

\medskip

Let us check the second point. Suppose that (\ref{eq:main}) holds for manifolds with abelian generic infinitesimal stabilizers. Let $K\rho$ be the regular 
admissible orbit such that $\QS_K(K\rho)$ is the trivial representation.

To any $\spinc$ manifold $(M,\Scal,K)$ with proper moment map $\Phi_\Scal$, we associate the product 
$M\times K\rho$ which is a $\spinc$ $K$-manifold with proper moment map $\Phi_{\Scal\boxtimes \Scal_{K\rho}}$. 
The multiplicative property (see Theorem \ref{theo:hochs-song}) gives 
$$
\qfor_K(M,\Scal)=\qfor_K(M\times K\rho,\Scal\boxtimes \Scal_{K\rho}).
$$
Now we remark that the $K$-manifold $M\times K\rho$ has abelian infinitesimal stabilizers, and that ${\rm p}\circ \Phi_{\Scal\boxtimes \Scal_{K\rho}}$ is proper if 
${\rm p}\circ \Phi_\Scal$ is proper. Then, when the moment map ${\rm p}\circ \Phi_\Scal$ is proper, we have
\begin{eqnarray*}
\qfor_K(M,\Scal)\vert_H&=&\qfor_K(M\times K\rho,\Scal\boxtimes \Scal_{K\rho})\vert_H\\
&=& \qfor_H(M\times K\rho,\Scal\boxtimes \Scal_{K\rho}) \hspace{21mm} [1]\\
&=& \qfor_H(M,\Scal)\otimes \Qcal_H(K\rho,\Scal_{K\rho})\hspace{15mm} [2]\\
&=&\qfor_H(M,\Scal).\hspace{42mm} [3]
\end{eqnarray*}
Here we see that $[1]$ is the identity (\ref{eq:main}) applied to $M\times K\rho$, $[2]$ is the multiplicative property relatively to the $H$-action, and 
$[3]$ is due to the fact that $\Qcal_H(K\rho,\Scal_{K\rho})$ is the trivial $H$-representation. $\Box$

\end{proof}

\subsection{De Concini-Procesi compactifications}\label{sec:Concini-Procesi}

We recall that $T$ is a maximal torus of the compact connected Lie group $K$, and $W$ is the corresponding Weyl group. 
We define a {\em $K$-adapted} polytope in $\tgot^*$ to be a $W$-invariant Delzant polytope $P$ in $\tgot^*$
whose vertices are regular elements of the weight lattice $\Lambda$. If $\{\lambda_1,\ldots,\lambda_r\}$
are the dominant weights lying in the union of all the closed one-dimensional faces of
$P$, then there is a $G \times G$-equivariant embedding of $G= K_\C$ into
$$\Pbb( \bigoplus_{i=1}^r (V^K_{\lambda_i})^* \otimes V^K_{\lambda_i})$$
associating to $g \in G$ its representation on $\bigoplus_{i=1}^r V^K_{\lambda_i}$. 
The closure $\Xcal_{P}$ of the image of $G$ in this
projective space is smooth and is equipped with a $K\times K$-action. The restriction of the canonical K\"{a}hler structure on 
$\Xcal_P$ defines a symplectic $2$-form $\Omega_{\Xcal_P}$.
We recall briefly the different properties of $(\Xcal_P,\Omega_{\Xcal_P})$ : all the details can be found in \cite{pep-IF}.

\begin{enumerate}
\item[(1)] $\Xcal_P$ is equipped with an Hamiltonian action of $K\times K$. Let 
$\Phi:=(\Phi_l,\Phi_r): \Xcal_P\to\kgot^*\times\kgot^*$ be the corresponding moment map.

\item[(2)] The image of $\Phi$ is equal to $\{(k\cdot\xi,-k'\cdot\xi)\ \vert\ \xi\in P\ 
{\rm and}\ k,k'\in K\}$.

\item[(3)] The Hamiltonian $K\times K$-manifold $(\Xcal_P,\Omega_{\Xcal_P})$ has no multiplicities: the pull-back by $\Phi$ of a 
$K\times K$-orbit in the image is a $K\times K$-orbit in $\Xcal_P$.

\item[(4)] The symplectic manifold $ (\Xcal_P,\Omega_{\Xcal_P})$ is prequantized by the restriction 
of the hyperplane line bundle $\Ocal(1)\to\Pbb( \oplus_{i=1}^N (V^K_{\lambda_i})^* \otimes V^K_{\lambda_i})$ to $\Xcal_P$: 
let us denoted $L_P$ the corresponding $K\times K$-equivariant line bundle. 
\end{enumerate} 

Let $\Ucal_P:= K\cdot P^\circ$ where $P^\circ$ is the interior of $P$. We define
$$
\Xcal_P^\circ:=\Phi_l^{-1}(\Ucal_P)
$$
which is an invariant, open and dense subset of $\Xcal_P$. We have the following 
important property concerning $\Xcal_P^\circ$.

\begin{enumerate}
\item[(5)] There exists an equivariant diffeomorphism $\Upsilon  : K\times \Ucal_P\to \Xcal_P^\circ$
such that $\Upsilon^*(\Phi_l)(g,\nu)= g\cdot \nu$ and $\Upsilon^*(\Phi_r)(g,\nu)=-\nu$.
\end{enumerate}

The manifold $\Xcal_P$ is equipped with  a family of $\spinc$ bundles 
$$
\Scal_P^n:= \bigwedge (\T \Xcal_P)^{1,0}\otimes L_P ^{\otimes n}, \quad n\geq 1,
$$ 
and we consider the corresponding $K\times K$-modules $\Qcal_{K\times K}(\Xcal_P,\Scal_P^n)$.

Let $\tgot^*_+\subset\tgot^*$ be a Weyl chamber and let 
$\Lambda\subset \tgot^*$ be the lattice of weights: we denote by $\rho\in \tgot^*_+$ the half sum of the positive roots.  

\begin{prop}\label{prop:Q-X-P}
We have the following decomposition
$$
\Qcal_{K\times K}(\Xcal_P,\Scal_P^n)=\sum_{\Ocal\cap \{n\,P^o+\rho\}\neq \emptyset}\ \pi_\Ocal\otimes \pi_{\Ocal^*} +
\sum_{\Ocal\cap \{n\,\partial P+\rho\}\neq \emptyset}a_{n,\Ocal}\ \pi_\Ocal\otimes \pi_{\Ocal^*}.
$$
\end{prop}
\begin{proof} 
The result is a consequence of the Meinrenken-Sjamaar $[Q,R]=0$ theorem \cite{Meinrenken98,Meinrenken-Sjamaar,Tian-Zhang98,pep-RR}.
To explain it, we parametrize the dual $\widehat{K}$ with the highest weights. For any dominant weight 
$\alpha\in \Lambda\cap \tgot^*_+$, let $V^K_\alpha$ be the irreducible representation of $K$ with highest weight $\alpha$. In other terms, 
$V^K_\alpha=\pi^K_{\alpha+\rho}$.

The symplectic $[Q,R]=0$ theorem tells us that the multiplicity of $V_\alpha^K\otimes V_\gamma^K$ in $\Qcal_{K\times K}(\Xcal_P,\Scal_P^n)$ is equal to the 
Riemann-Roch number of the symplectic reduced space $\Phi^{-1}(K\frac{\alpha}{n}\times K\frac{\gamma}{n})/K\times K$.

Points $(2)$ above tell us that $\Phi^{-1}(Ka\times Kb)/K\times K$ is non empty only if $Kb=-Ka$ and $Ka\cap P\neq \emptyset$. With point $(4)$ , we see that 
the reduced space $\Phi^{-1}(Ka\times -Ka)/K\times K$ is a (smooth) point if $a\in P^o$. The proof is completed.
\end{proof}

\subsection{Cutting}

Let $M$ be a $K$-manifold of even dimension, oriented, equipped with a $K$-equivariant $\spinc$ bundle $\Scal$. Let $\Phi_\Scal$ be the moment map associated to a hermitian connection on $\det(\Scal)$.  We assume that $\Phi_\Scal$ is {\em proper}.

We consider the manifold $\Xcal_P$ equipped with the $\spinc$ bundle 
$\Scal_P^n:= \bigwedge (\T \Xcal_P)^{1,0}\otimes L_P ^{\otimes n}$. The determinant line bundle
$\det(\Scal_P^n)$ is equal to $$\det((\T \Xcal_P)^{1,0})\otimes L_P ^{\otimes 2n}.$$

Let $\varphi_l,\varphi_r:\Xcal_P\to\kgot^*$ be the moment maps associated to the action of $K\times K$ 
on the line bundle $\det((\T \Xcal_P)^{1,0})$. So the moment map 
relative to the action of $K\times K$ on $\det(\Scal_P^n)$ is the map $\Phi^n=(\Phi^n_l,\Phi^n_r): \Xcal_P\to\kgot^*\times \kgot^*$ 
defined by $\Phi^n_l= n\Phi_l +\varphi_l$ and $\Phi^n_r= n\Phi_r +\varphi_r$.

On the dense open subset $\Xcal_P^o\simeq K\times \Ucal_P$ the line bundle $\det((\T \Xcal_P)^{1,0})$ admits a trivialization. 
Let $c>0$ such that the closed ball $\{\|\xi\|\leq c\}$ is contained in $\Ucal_P$. By partition of unity, we can choose a connection on 
$\det((\T \Xcal_P)^{1,0})$ such that the corresponding moment maps satisfy
\begin{equation}\label{eq:varphi=0}
\varphi_l(x)=\varphi_r(x)=0
\end{equation}
if $x\in \Xcal_P$ satisfies $\|\Phi_r(x)\|\leq c$.

We consider now the product $M\times \Xcal_P$ equipped with the $\spinc$ bundle
$\Scal\boxtimes\Scal_P^n$ and with the following $K\times K$ action:
$$
(k_l,k_r)\cdot (m,x)= (k_r\cdot m, (k_l,k_r)\cdot x).
$$
The moment map relative to the action of $K\times K$ on the line bundle $\det(\Scal\boxtimes\Scal_P^n)$ is 
$$
(m,x)\longmapsto (\Phi^n_l(x), \Phi_\Scal(m)+\Phi^n_r(x)).
$$

We restrict the action of $K\times K$ on $M\times \Xcal_P$ to the subgroup $H\times K$. We see that the corresponding moment map
$(p\circ \Phi^n_l, \Phi_\Scal+\Phi^n_r)$ is proper, so we can consider the formal geometric quantization of the $\spinc$ manifold 
$(M\times \Xcal_P,\Scal\boxtimes\Scal_P^n)$ relative to the action of $H\times K$: 
$\qfor_{K\times H}(M\times \Xcal_P,\Scal\boxtimes\Scal_P^n)\in \hat{R}(K\times H)$.

We are interested in the following $H$-module 
$$
E(n):= \left[\qfor_{H\times K}(M\times \Xcal_P,\Scal\boxtimes\Scal_P^n)\right]^K.
$$

The proof of Theorem \ref{theo-principal} will follows from the computation of the asymptotic behaviour of $E(n)$ by two means.

\subsection{First computation}

Let us write $\qfor_{K}(M,\Scal)= \sum_{\Ocal} m_\Ocal\,\pi_\Ocal^K$.

\begin{prop}\label{prop:first-computation}We have  
$$
\lim_{n\to\infty} E(n)=\qfor_{K}(M,\Scal)\vert_H:=\sum_{\lambda\in\hat{K}} m_\Ocal\, \pi_\Ocal^K\vert_H \in \hat{R}(H)
$$
\end{prop}

\begin{proof}
We start by using the multiplicative property (Theorem \ref{theo:hochs-song}):
$$
\qfor_{H\times K}(M\times \Xcal_P,\Scal\boxtimes\Scal_P^n)= \qfor_{K}(M,\Scal_M)\otimes \Qcal_{H\times K}(\Xcal_P,\Scal_P^n).
$$
Thanks to Proposition \ref {prop:Q-X-P}, we know that $\Qcal_{H\times K}(\Xcal_P,\Scal_P^n)$ is equal to 
$$
\sum_{\Ocal\cap \{n P^\circ +\rho\}\neq\emptyset } \pi^K_\Ocal\vert_H\otimes (\pi^K_\Ocal)^* + R(n)
$$
with $R(n)= \sum_{\Ocal} a_{n,\Ocal}\, \pi^K_\Ocal\vert_H\otimes (\pi^K_\Ocal)^*$ where $a_{n,\Ocal}\neq 0$ only if $\Ocal\cap \{n \partial P +\rho\}\neq\emptyset$. 
So we see that 
$$
E(n)=  \sum_{\Ocal\cap \{n P^\circ +\rho\}\neq\emptyset  } m_\Ocal\,\pi^K_\Ocal\vert_H + r(n)
$$
with $ r(n)=\left[\qfor_{K}(M,\Scal)\otimes R(n)\right]^K$. It remains to check that $\lim_{n\to\infty} r(n)= 0$ in $\hat{R}(H)$.

If $\Ocal$ is a $K$-coadjoint orbit we denote $\|\Ocal\|$ the norm of any of its element. Note that there exists $d>0$ such that if 
$\Ocal\cap \{n \partial P +\rho\}\neq\emptyset$, then $\|\Ocal\|\geq n d$.

Let $\Ocal'$ be a regular admissible $H$-orbit. By definition the multiplicity of $\pi_{\Ocal'}^H$ in $r(n)$ decomposes as follows 
$$
[\pi_{\Ocal'}^H: r(n)]= \sum_{\Ocal} a_{n,\Ocal}\, m_\Ocal\,  [\pi_{\Ocal'}^H: \pi_\Ocal^K\vert_H].
$$

Suppose that $r(n)$ does not tends to $0$ in $\hat{R}(H)$: there exists a regular admissible $H$-orbit $\Ocal'$ such that the set 
$\{n\geq 1, [\pi_{\Ocal'}^H: r(n)]\neq 0\}$ is infinite. Hence there exists a sequence 
$(n_k, \Ocal_k)$ such that $\lim_{k\to \infty} n_k=\infty$ and $a_{n_k,\Ocal_k}\, m_{\Ocal_k}\,  [\pi_{\Ocal'}^H: \pi_{\Ocal_k}^K\vert_H]\neq 0$.

Thanks to the $[Q,R]=0$ property, we have 
\begin{equation}
\begin{cases}
1.\ \|\Ocal_k\|\geq d\, n_k, \\
2.\ \Ocal_k\in \Phi_\Scal(M),\\
3.\ \Ocal'\subset {\rm p}(\Ocal_k),
\end{cases}
\end{equation}
where ${\rm p}:\kgot^*\to \hgot^*$ is the projection. Points $2.$ and $3.$ give that 
$$
\Ocal_k\in K\Phi_\Scal\left(({\rm p}\circ \Phi_\Scal)^{-1}(\Ocal')\right).
$$
Since ${\rm p}\circ \Phi_\Scal$ is proper, we have only a finite number of $K$-admissible orbits contained in the compact set 
$K\Phi_\Scal\left(({\rm p}\circ \Phi_\Scal)^{-1}(\Ocal')\right)$. This is in contradiction with the first point. $\Box$

\end{proof}

\subsection{Reduction in stage}\label{sec:stage}

In this section, we explain the case of reduction in stages. Suppose that we have an action of the compact Lie group
$G\times K$ on the $\spinc$ manifold $(N,\Scal_N)$. Let $\Phi_{\Scal_N}=\Phi_{\Scal_N}^G\oplus\Phi_{\Scal_N}^K : N\to \ggot^*\oplus\kgot^*$ be  the
corresponding moment map associated to the choice of an invariant connection $\nabla$ on $\det(\Scal_N)$. We suppose that
\begin{itemize}
\item  $0$ is a regular value of $\Phi_{\Scal_N}^K$,
\item  $K$ acts freely on $Z:=(\Phi_{\Scal_N}^K)^{-1}(0)$,
\item the set $\Phi_{\Scal_N}^{-1}(0)$ is compact.
\end{itemize}

We denote by $\pi: Z\to N_0:=Z/K$ the corresponding $G$-equivariant principal fibration.

On $Z$, we obtain an exact sequence $0\longrightarrow \T Z\longrightarrow \T M\vert_Z \stackrel{\T\Phi_K}{\longrightarrow} [\kgot^*]\to 0$,
where $[\kgot^*]$ is the trivial bundle $Z\times\kgot^*$. We have also an orthogonal decomposition
$\T Z= \T_{K} Z \oplus [\kgot]$ where $[\kgot]$ is the sub-bundle identified to $Z\times\kgot$ through the map
$(p,X)\mapsto X\cdot p$. So $\T M\vert_Z$ admits the orthogonal decomposition
$\T N\vert_Z \simeq \T_{K} Z\oplus [\kgot] \oplus  [\kgot^*]$. We rewrite this as
\begin{equation}\label{eq:tangent-P}
\T N\vert_Z \simeq \T_{K} Z\oplus [\kgot_\C]
\end{equation}
with the convention $[\kgot]=Z\times(\kgot\otimes\R) $ and $[\kgot^*] = Z\times (\kgot\otimes i\R)$.
Note that the bundle $\T_{K} Z$ is naturally identified with $\pi^*(\T N_{0})$.

We can divide the $\spinc$-bundle $\Scal_N\vert_Z$ by the $\spinc$-bundle
$\bigwedge \kgot_\C$ for the vector space $\kgot_\C$ (see Section 2.2 in \cite{pep-vergne:witten}).

\begin{defi}\label{def:E-red}
Let $\Scal_{N_0}$ be the $\spinc$-bundle on $N_0$ such that
$$
\Scal_N\vert_Z \simeq \pi^*(\Scal_{N_0})\otimes [\bigwedge\kgot_\C]
$$
is an isomorphism of graded Clifford  bundles on $\T N\vert_Z$. 
\end{defi}

We see then that the line bundle $\det(\Scal_{N_0})$ is equal to  $\det(\Scal_{N})\vert_Z/K$. Hence the connection $\nabla$ on $\det(\Scal_N)$ 
induces a $G$-invariant connection $\nabla_0$ on $\det(\Scal_{N_0})$. The corresponding moment map $\Phi_{\Scal_{N_0}}:N_0\to \ggot^*$ is the equivariant map 
induced by $\Phi_{\Scal_N}^G: N\to \ggot^*$.

\begin{prop}\label{prop:localisation-stage}
We have the following relation
$$
\left[\Qcal_{G\times K}(N,\Scal_N,\{\Phi_{\Scal_N}=0\})\right]^K=
\Qcal_G(N_{0},\Scal_{N_0},\{\Phi_{\Scal_{N_0}}=0\}) \qquad \mathrm{in}\quad \hat{R}(G).
$$
\end{prop}

\begin{proof}
The proof is done in Section 3.4.2 of \cite{pep-vergne:witten} in the Hamiltonian setting. The same proof works here. 
\end{proof}

\subsection{Second computation}

We consider $\qfor_H(M,\Scal)\in \hat{R}(H)$. 

\begin{prop}\label{prop:second-computation}
Suppose that the generic infinitesimal stabilizer of the $K$-action on $M$ is {\em abelian}. Let $\Ocal'$ be a regular admissible $H$-orbit. There exists $n_{\Ocal'}\geq 1$ such that 
$$
\left[\pi_{\Ocal'}^H : E(n)\right]=\left[\pi_{\Ocal'}^H :\qfor_H(M,\Scal)\right]
$$
when $n\geq n_{\Ocal'}$.
\end{prop}

\begin{proof} First of all, since the $K$-action on $M$ has  generic abelian infinitesimal stabilizers, 
we see that the $H\times K$-action on $M\times \Xcal_P$ has also generic abelian infinitesimal stabilizers.

Let us denote $\tilde{\Ocal}$ the $H\times K$ regular admissible orbit $\Ocal'\times K\rho$. We work with the $H\times K$ manifold 
$$
N:= M\times \Xcal_P \times \tilde{\Ocal}^*
$$
which is equipped with the $\spinc$ bundles 
$\Scal_N^n:= \Scal \boxtimes\Scal_P^n\boxtimes \Scal_{\tilde{\Ocal}^*}$. The moment map associated to the action of 
$H\times K$ on $\det(\Scal_N^n)$ is $\Phi_{\Scal_N^n}=(\Phi^n_H,\Phi^n_K)$ where 
$$
\Phi^n_H(m,x,\eta,\xi)={\rm p}\left(n\Phi_l(x)+\varphi_l(x)\right)+\eta,
$$
and
$$
\Phi^n_K(m,x,\eta,\xi)=\Phi_\Scal(m)+n\Phi_r(x)+ \varphi_r(x)+\xi
$$
for $(m,x,\eta,\xi)\in M\times \Xcal_P\times (\Ocal')^*\times (K\rho)^*$.

Thanks to the multiplicative property we have
$$
\left[\pi_{\Ocal'}^H : E(n)\right]= \left[\qfor_{H\times K}(N,\Scal_N^n)\right]^{H\times K}.
$$
Using the fact that the $H\times K$-action on $M\times \Xcal_P$ has generic abelian infinitesimal stabilizers, we know that 
\begin{equation}\label{eq:multiplicity-localized}
\left[\pi_{\Ocal'}^H : E(n)\right]= \left[\Qcal_{H\times K}(N,\Scal_N^n,\{\Phi_{\Scal_N^n}=0\})\right]^{H\times K}.
\end{equation}
See Remark \ref{rem:abelian-stabilizer}. Now we are going to compute the right hand side of (\ref{eq:multiplicity-localized}) 
by using the reduction in stage (see Section \ref{sec:stage}).

We start with the

\begin{lem}
There exists $R,R'>0$, independent of $n$, such that 
if \break $(m,x,\eta,\xi)\in \{\Phi_{\Scal_N^n}=0\}$ then
$\|\Phi_\Scal(m)\|\leq R$ and $ \|\Phi_r(x)\|\leq R'/n$.
\end{lem}
\begin{proof} 
Let $(m,x,\eta,\xi)\in \{\Phi_{\Scal_N^n}=0\}$. We have ${\rm p}\left(n\Phi_l(x)+\varphi_l(x)\right)+\eta=0$ and 
$\Phi_\Scal(m)+n\Phi_r(x)+ \varphi_r(x)+\xi=0$. Let $k\in K$ such that $k\Phi_r(x)=-\Phi_l(x)$ 
(see Point $(2)$ in Section \ref{sec:Concini-Procesi}). We get then 
${\rm p}(\Phi_\Scal(km))+ {\rm p}(k\varphi_r(x)+\varphi_l(x)+\xi)+\eta=0$. The term 
${\rm p}(k\varphi_r(x)+\varphi_l(x)+\xi)+\eta$ is bounded, and since ${\rm p}\circ\Phi_\Scal$ is proper, the variable $m$ belongs to a compact 
of $M$ (independent of $n$). Finally the identity $\Phi_\Scal(m)+n\Phi_r(x)+ \varphi_r(x)+\xi=0$ shows that $n\Phi_r(x)$ is bounded by a quantity independent of $n$. $\Box$
\end{proof}

\medskip

So, if $n$ is large enough, the set $\{\Phi_{\Scal_N^n}=0\}$ is contained in the open subset $M\times  \Xcal_P^o\times\tilde{\Ocal}^*\subset N$ that we can identify with
$$
\tilde{N}= M \times K\times \Ucal_P \times \tilde{\Ocal}^*
$$
through the diffeomorphism $\Upsilon : K\times \Ucal_P\to \Xcal_P^o$ (see Point $(4)$ in Section \ref{sec:Concini-Procesi}). 
Moreover, thanks to (\ref{eq:varphi=0}), for $n$ large enough an element
$(m,g,\nu,\eta,\xi)\in \tilde{N}$ belongs to $\{\Phi_{\Scal_N^n}=0\}$ if and only if 
\begin{equation}
\begin{cases}
n{\rm p}(g\nu)+\eta=0,\\
\Phi_\Scal(m) -n\nu+\xi=0.
\end{cases}
\end{equation}

We use now the reduction in stage for $n$ large enough. The map 
$$
(m,\eta,\xi,g)\mapsto (m,g,\frac{\Phi_\Scal(m)+\xi}{n},\eta)
$$ 
defines a diffeomorphism between  $M\times \tilde{\Ocal}^*\times K$ and the sub-manifold \break 
$Z:=\{\Phi_K^n=0\}\subset \tilde{N}$ and induces a diffeomorphism
$$
\Psi: M\times \tilde{\Ocal}^* \stackrel{\sim}{\longrightarrow} N_0= Z/K.
$$
Through the diffeomorphism $\Psi$, the $H$-action on $N_0= Z/K$ corresponds to the induced action of the subgroup 
$H\simeq \{(h,h), h\in H\}\subset H\times K$ on $M\times \tilde{\Ocal}^*$.
Through the diffeomorphism $\Psi$, the moment map $\Phi_{\Scal^n_{N_0}}: N_0\to \hgot^*$ becomes
$$
\Phi_{\tilde{\Ocal}}(m,\xi,\eta)={\rm p}(\Phi_\Scal(m)+\xi)+\eta, 
$$
for $(m,\xi,\eta)\in M\times \tilde{\Ocal}^*$.

\begin{lem}\label{lem:diffeo-psi}
Through the diffeomorphism $\Psi$, the induced $\spinc$ bundle $\Scal^n_{N_0}$ corresponds to $\Scal_M\boxtimes \Scal_{ \tilde{\Ocal}^*}$.
\end{lem}

\begin{proof}
We consider the restriction of the $\spinc$ bundle $\Scal_P^n$ to the open subset $\Xcal^\circ_P$. Let 
$\Scal^n:=\Upsilon^{-1}(\Scal_P^n\vert_{\Xcal^\circ_P})$ be the corresponding $K\times K$-equivariant $\spinc$ bundle on 
$K\times\Ucal_P$. It must be of the form $\Scal^n\simeq F\times \bigwedge\kgot_\C\times K\times\Ucal_P$ where $F$ is a character of 
$K\times K$. If we look at the value of $\Scal_P^n$ at the point $\Upsilon(1,0)\in\Xcal_P$, we see that $F$ is trivial. The Lemma follows. 
$\Box$
\end{proof}

\medskip

Finally, for $n$ large enough, we get
\begin{eqnarray*}
\left[\pi_{\Ocal'}^H : E(n)\right]&=& \left[\Qcal_{H\times K}(N,\Scal_N^n,\{\Phi_{\Scal_N^n}=0\})\right]^{H\times K}
\hspace{13mm} [1]\\
&=&\left[\Qcal_{H}(N_0,\Scal_{N_0}^n,\{\Phi_{\Scal^n_{N_0}}=0\})\right]^{H}\hspace{19mm} [2]\\
&=&\left[\Qcal_{H}(M\times\tilde{\Ocal}^*,\Scal\boxtimes\Scal_{\tilde{\Ocal}^*},\{\Phi_{\tilde{\Ocal}}=0\})\right]^{H} \hspace{5mm} [3]\\
&=& \left[\qfor_{H}(M\times\tilde{\Ocal}^*,\Scal\boxtimes\Scal_{\tilde{\Ocal}^*})\right]^{H}\hspace{21mm} [4]\\
&=& \left[\qfor_{H}(M,\Scal)\otimes\Qcal_H(\tilde{\Ocal}^*,\Scal_{\tilde{\Ocal}^*})\right]^H\hspace{15mm} [5]\\
&=& \left[\pi_{\Ocal'}^H : \qfor_{H}(M,\Scal)\right].\hspace{36mm} [6]
\end{eqnarray*}

First we see that $[1]$ corresponds to (\ref{eq:multiplicity-localized}). 
Equality $[2]$ is the reduction in stage (see Proposition \ref{prop:localisation-stage}) and Equality $[3]$ is a 
consequence of the diffeomorphism $\Psi$ (see Lemma \ref{lem:diffeo-psi}). Equality $[4]$ follows from the fact that $M$ has 
abelian generic infinitesimal stabilizers (see Remark \ref{rem:abelian-stabilizer}). 
Equality $[5]$ is a consequence of the multiplicative property. Equality $[6]$ follows from the identity 
$\Qcal_H(\tilde{\Ocal}^*,\Scal_{\tilde{\Ocal}^*})=(\pi_{\Ocal'}^H)^*$.

The proof of Proposition \ref{prop:second-computation} is completed. $\Box$
\end{proof}

\bigskip

We can now conclude our exposition.  Proposition \ref{prop:first-computation} tell us that \break 
$\lim_{n\to\infty} E(n)=\qfor_{K}(M,\Scal)\vert_H$
while Proposition \ref{prop:second-computation} says that $\lim_{n\to\infty} E(n)$ $=\qfor_{H}(M,\Scal)$ 
when the manifold $M$ has abelian generic infinitesimal stabilizers. 
So we have proved Theorem \ref{theo-principal} for manifolds with abelian generic infinitesimal stabilizers. 
But we have checked in Lemma \ref{lem:simple} that 
it is sufficient to get the proof of Theorem \ref{theo-principal} in the general case.


\end{document}